%% file: psuedoknots10Oct2014.tex
\documentclass[11pt]{amsart}
\usepackage{pifont}
\usepackage{amsthm}
\usepackage{amsfonts}
\usepackage{amssymb}
\usepackage[mathscr]{euscript}
\usepackage[all]{xy}
\usepackage[dvips]{graphics}
\usepackage{amsmath}
\usepackage[dvips,final]{graphicx}
\usepackage[dvips]{geometry}
\usepackage{color}
\usepackage{epsfig}
\usepackage{latexsym}
\usepackage{subcaption}

\graphicspath{{diagrams/}}

\newtheorem{theorem}{Theorem}

\newtheorem{corollary}[theorem]{Corollary}

\theoremstyle{definition}

\newtheorem{example}{Example}[section]

\theoremstyle{remark}

\title[Psuedo knots and cosmetic crossings]{Pseudo knots and an obstruction to cosmetic crossings}
\author{Heather A. Dye}

\begin{document}
\begin{abstract} 
Pseudo links have two crossing types: classical crossings and indeterminate crossings. They were first introduced by Ryo Hanaki as a possible tool for analyzing images produced by electron microscopy of DNA. A normalized bracket polynomial is defined for pseudo links and then used to construct and obstruction to cosmetic crossings in classical links. 
\end{abstract}
\maketitle

\section{Introduction}

The set of psuedo knots and links was first introduced by Hanaki Ryo \cite{MR2768805}, 
to study the type of diagrams produced electron microscopy of DNA.
In these images, the over-under crossing information is often blurred; this results in 
a diagram with classical crossings and crossings where the under-over crossing information is unknown. Based on this physical 
interpretation, Ryo developed a set of Reidemeister-like moves that are not dependent on crossing type. Subsequent work by Allison Henrich explored several invariants of 
pseudo knots \cite{MR3298207}, \cite{Allison-PseudoBounds}, \cite{MR3084750}.

In this paper, we recall the definition of pseudo knots and links. Then, a modification of the bracket polynomial is defined for pseudo links. Finally, the pseudo bracket is applied to produce an obstruction to cosmetic crossings. 
A classical crossing $x$ in a  knot diagram $D$ is said to be cosmetic if $D$ is equivalent to 
the knot diagram $D'$ where $D'$ is obtained by switching the crossing $x$ from a 
positively signed crossing to a negatively signed crossing (or vice versa). 
X. S. Lin conjectured that cosmetic crossings do not exist (with limited exceptions such as Reidemeister I twists and nugatory crossings). This is problem 1.58 on Kirby's problem list \cite{kirby}.

A pseudo link diagram $D$ is a decorated immersion of $n$ oriented copies of $S^1$ with two types of crossings. 
A crossing is either a classical crossing with over-under markings or a pseudo-crossing that is marked by a solid square as shown in Figure \ref{fig:crossings}. 

\begin{figure}[htb]
\begin{subfigure}{0.32\linewidth}
\[ \begin{array}{c} \scalebox{0.5}{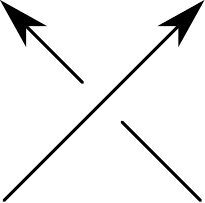}
\end{array} \]
\caption{Positive crossing}
\label{fig:poscrossing}
\end{subfigure} 
\begin{subfigure}{0.32\linewidth}
\[ \begin{array}{c} \scalebox{0.5}{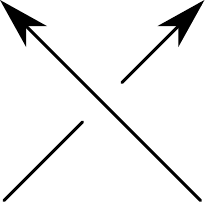}
\end{array} \]
\caption{Negative crossing}
\label{fig:negcrossing}
\end{subfigure} 
\begin{subfigure}{0.32\linewidth}
\[ \begin{array}{c} \scalebox{0.5}{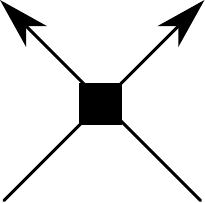}
\end{array} \]
\caption{Pseudo crossing}
\label{fig:pseudocrossing}
\end{subfigure} 
\caption{Crossing types}
\label{fig:crossings}
\end{figure}
Classical crossings follow the usual sign conventions. For a positive crossing $c$, $ \mathit{sgn} (c) = +1$ and for a negative crossing $c$, $ \mathit{sgn} (c) = -1$. 
 
Two pseudo link diagrams are  equivalent if they are related by a sequence of Reidemeister moves (Figure \ref{fig:rmoves}) and Pseudo moves (Figure \ref{fig:pmoves}).  A pseudo link is an equivalence class of pseudo link diagrams. 

\begin{figure}[htb]
\begin{subfigure}{0.48\linewidth}
\[ \begin{array}{c} \scalebox{0.5}{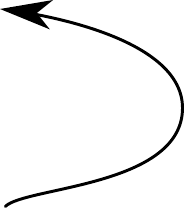}
\end{array} \leftrightarrow 
\begin{array}{c} \scalebox{0.5}{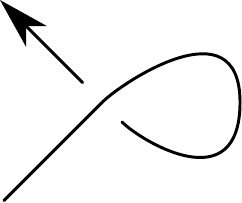} \end{array}  \]
\caption{Reidemeister I} 
\label{fig:r1move}
\end{subfigure} 
\begin{subfigure}{0.48\linewidth}
\[ \begin{array}{c} \scalebox{0.5}{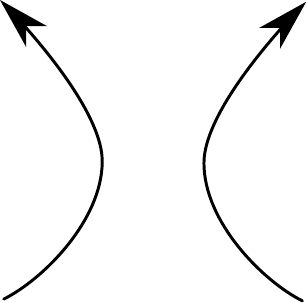}
\end{array} \leftrightarrow 
\begin{array}{c} \scalebox{0.5}{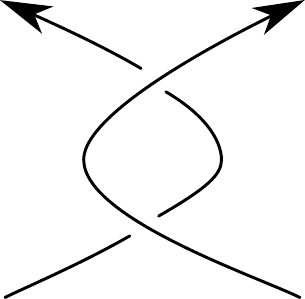} \end{array}  \]
\caption{Reidemeister II} 
\label{fig:r2move}
\end{subfigure} 
\begin{subfigure}{0.48\linewidth}
\[ \begin{array}{c} \scalebox{0.5}{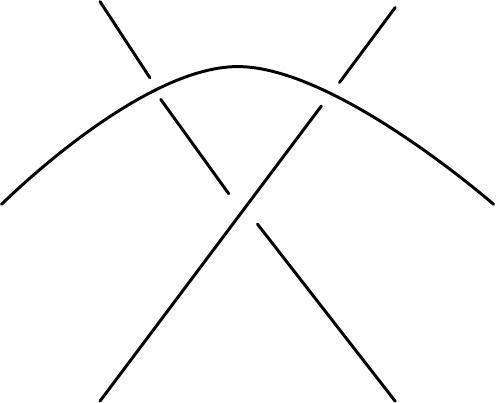}
\end{array} \leftrightarrow 
\begin{array}{c} \scalebox{0.5}{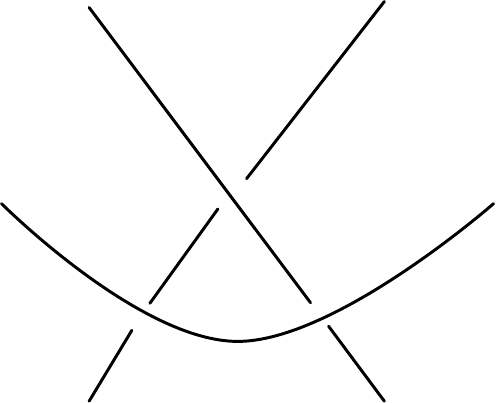} \end{array}  \]
\caption{Reidemeister III} 
\label{fig:r3move}
\end{subfigure} 
\caption{Reidemeister moves} 
\label{fig:rmoves}
\end{figure} 
\begin{figure}[htb]
\begin{subfigure}{0.48\linewidth}
\[ \begin{array}{c} \scalebox{0.5}{\input{diagrams/r1movelhs.pdf_tex}}
\end{array} \leftrightarrow 
\begin{array}{c} \scalebox{0.5}{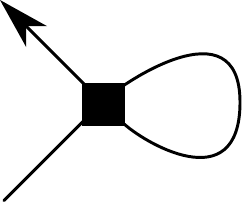} \end{array}  \]
\caption{Pseudo I} 
\label{fig:p1move}
\end{subfigure} 
\begin{subfigure}{0.48\linewidth}
\[ \begin{array}{c} \scalebox{0.5}{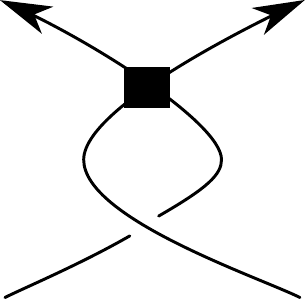}
\end{array} \leftrightarrow 
\begin{array}{c} \scalebox{0.5}{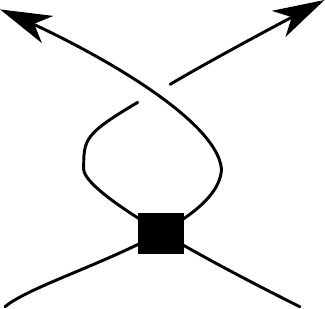} \end{array}  \]
\caption{Pseudo II} 
\label{fig:p2move}
\end{subfigure} 
\begin{subfigure}{0.48\linewidth}
\[ \begin{array}{c} \scalebox{0.5}{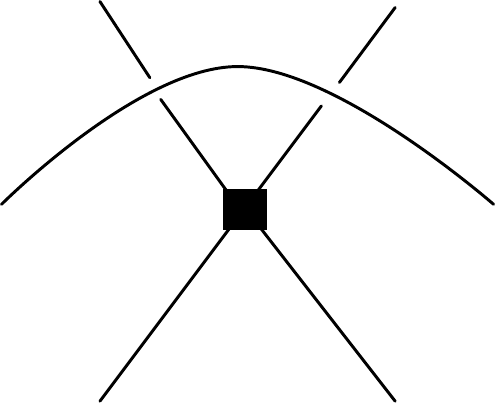}
\end{array} \leftrightarrow 
\begin{array}{c} \scalebox{0.5}{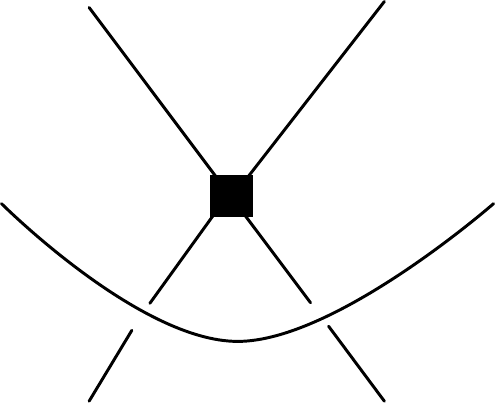} \end{array}  \]
\caption{ Pseudo III} 
\label{fig:p3move}
\end{subfigure} 
\caption{Pseudo  moves} 
\label{fig:pmoves}
\end{figure}

\section{The pseudo bracket polynomial}

Let $D$ be an oriented pseudo link. The pseudo bracket polynomial is defined by a skein relation, building on the definition of the Kauffman bracket polynomial \cite{kauffmanmaa}.
We expand a positive crossing as:
\begin{equation}\label{skeinpos}
\Bigg \langle  \begin{array}{c} \scalebox{0.3}{\input{diagrams/positivecrossing.pdf_tex}} \end{array} 
\Bigg \rangle = A \Bigg \langle  \begin{array}{c} \scalebox{0.3}{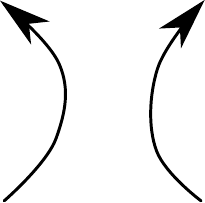} \end{array} 
\Bigg \rangle + A^{-1} \Bigg  \langle   \begin{array}{c} \scalebox{0.3}{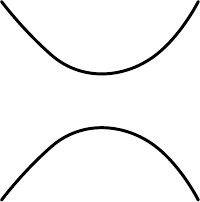} \end{array} 
\Bigg \rangle .
\end{equation}
A negative crossing is expanded as:
\begin{equation}\label{skeinneg}
\Bigg \langle \begin{array}{c} \scalebox{0.3}{\input{diagrams/negativecrossing.pdf_tex}}  \end{array}
\Bigg \rangle = A \Bigg \langle  \begin{array}{c} \scalebox{0.3}{\input{diagrams/vsmooth.pdf_tex}} \end{array}
\Bigg \rangle + A^{-1} \Bigg  \langle  \begin{array}{c} \scalebox{0.3}{\input{diagrams/hsmooth.pdf_tex}} \end{array}
\Bigg \rangle .
\end{equation} 
Pseudo crossings are expanded as
\begin{equation} \label{skeinpseudo}
\Bigg \langle  \begin{array}{c}\scalebox{0.3}{\input{diagrams/psuedocrossing.pdf_tex}} \end{array}
\Bigg \rangle = V \Bigg \langle  \begin{array}{c}\scalebox{0.3}{\input{diagrams/vsmooth.pdf_tex}} \end{array}
\Bigg \rangle + H \Bigg  \langle  \begin{array}{c} \scalebox{0.3}{\input{diagrams/hsmooth.pdf_tex}} \end{array} 
\Bigg \rangle 
\end{equation}
where $H = 1 - V d$. 

  Let $U$ denote the unknot and let $d = -A^2 -A^{-2}$. Expanding all crossings in a pseudo diagram results in a collection of simple closed curves. We evaluate a bracket containing an unlinked, simple closed curve using the simplifications:
\begin{equation} \label{reducebracket} 
\langle U \rangle  = 1 \text{ and }
 \langle U \cup K \rangle = d \langle K \rangle  .
\end{equation}

The set of classical crossings in a link $K$ is denoted as $ \mathit{C} (K)$. The writhe of $K$ is defined as 
\begin{equation} \label{writhe}
\mathit{w} (K) = \sum_{c \in \mathit{C} (K)} \mathit{sgn}(c).
\end{equation} 
Then, the  normalized pseudo bracket of a pseudo link $K$ is
\begin{equation}
P_K (A,V) = (-A^{-3})^{\mathit{w}(K)} \langle K \rangle .
\end{equation}

\begin{theorem}For all pseudo links $K$, the pseudo bracket is invariant under Reidemeister moves II and III and the pseudo moves. \end{theorem}
\begin{proof} 
Invariance under the Reidemeister moves is immediate. The skein relation on a classical link diagram gives the Kauffman bracket polynomial \cite{kauffmanmaa}. For a classical link $K$, 
$P_K (A,V) = f_K (A)$.  

We show that the bracket is invariant under the Pseudo moves. We begin with the Pseudo III move.
\begin{align*}
\Bigg \langle \begin{array}{c} \scalebox{0.3}{ 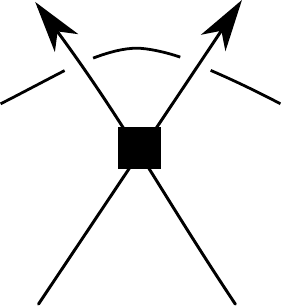}
\end{array} \Bigg \rangle &=
V \Bigg \langle \begin{array}{c} \scalebox{0.3}{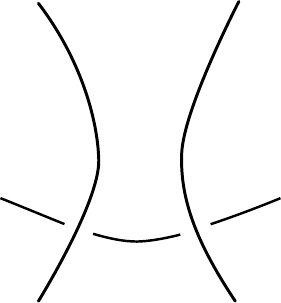}
\end{array} \Bigg \rangle
+ H \Bigg \langle \begin{array}{c} \scalebox{0.3}{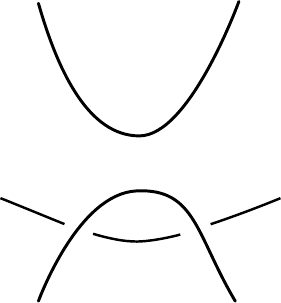}
\end{array} \Bigg \rangle  \\
&= V \Bigg \langle \begin{array}{c} \scalebox{0.3}{\input{diagrams/pseudor3vsmooth.pdf_tex}}
\end{array} \Bigg \rangle
+ H \Bigg \langle \begin{array}{c} \scalebox{0.3}{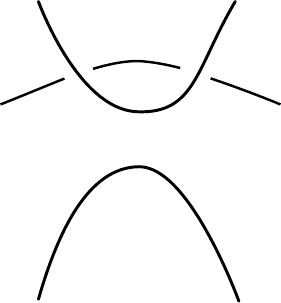}
\end{array} \Bigg \rangle  \\
&= \Bigg \langle \begin{array}{c} \scalebox{0.3}{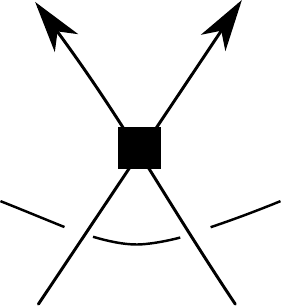}
\end{array} \Bigg \rangle .
\end{align*}
Next, the Pseudo II move. 
\begin{align*}
\Bigg \langle \begin{array}{c} \scalebox{0.3}{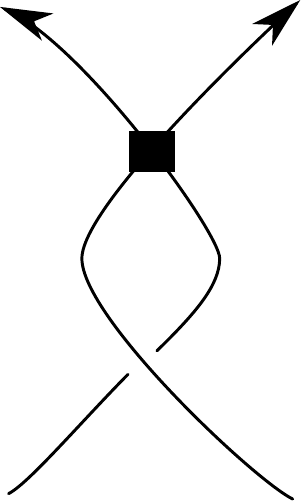}
\end{array} \Bigg \rangle &= 
V \Bigg \langle \begin{array}{c} \scalebox{0.4}{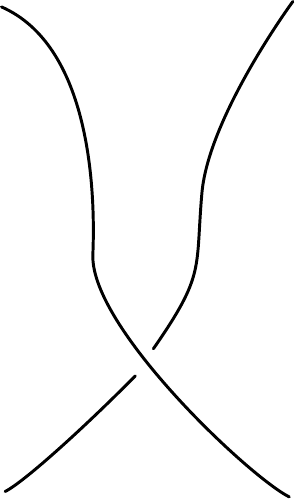}
\end{array} \Bigg \rangle  
+ H \Bigg \langle \begin{array}{c} \scalebox{0.4}{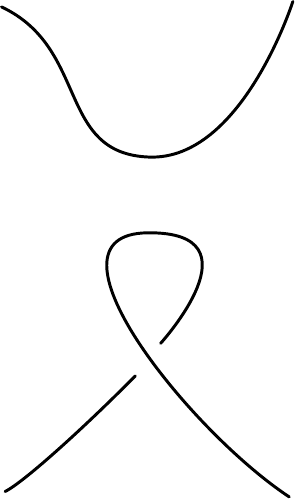}
\end{array} \Bigg \rangle \\
&= 
V \Bigg \langle \begin{array}{c} \scalebox{0.4}{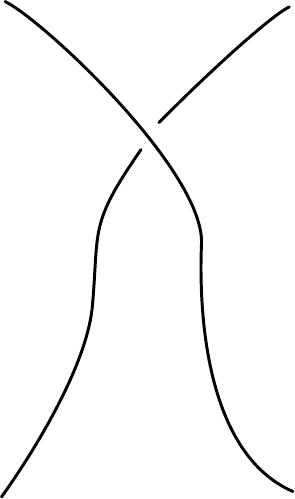}
\end{array} \Bigg \rangle  
+ H \Bigg \langle \begin{array}{c} \scalebox{0.4}{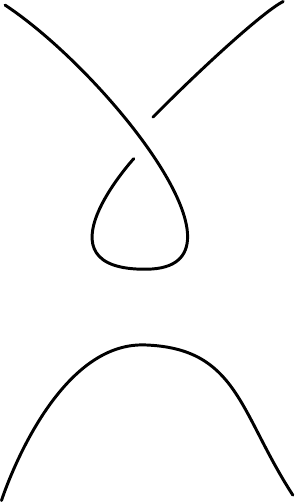}
\end{array} \Bigg \rangle  \\
&= \Bigg \langle \begin{array}{c} \scalebox{0.3}{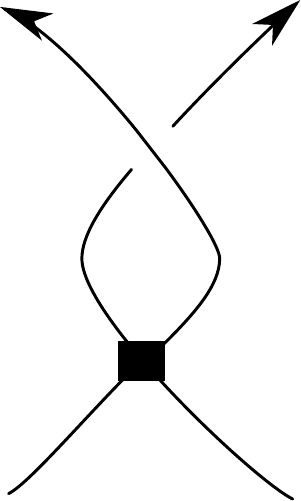}
\end{array} \Bigg \rangle .
\end{align*}

In the Pseudo I move,
\begin{align*}
\Bigg \langle \begin{array}{c} \scalebox{0.3}{\input{diagrams/pr1moverhs.pdf_tex}}
\end{array} \Bigg \rangle &=
V \Bigg \langle \begin{array}{c} \scalebox{0.3}{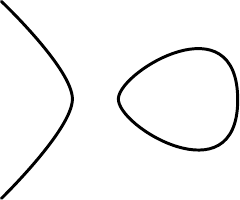}
\end{array} \Bigg \rangle + H \Bigg \langle \begin{array}{c} \scalebox{0.3}{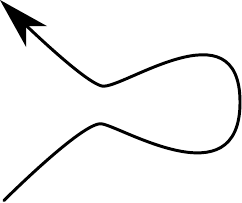} \end{array} \Bigg \rangle \\
&= \left( V d + H \right) \Bigg \langle 
\begin{array}{c} \scalebox{0.3}{\input{diagrams/r1movelhs.pdf_tex}} \end{array} \Bigg \rangle .
\end{align*}
Then $ 1= Vd + H$ or $H = 1- V d$. 
\end{proof}

\begin{corollary} For all pseudo links $K$, $P_K (A,V) $ is invariant under the Pseudo moves and the Reidemeister moves. \end{corollary} 

\begin{example}
The pseudo bracket is a applied to a trefoil with one pseudo crossing. This pseudo diagram is denoted as $PT$. 
\begin{align*}
\Bigg \langle \begin{array}{c} \scalebox{0.3}{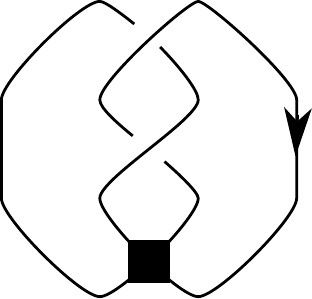}
\end{array} \Bigg \rangle
&= A^2 \Bigg \langle  \begin{array}{c} \scalebox{0.3}{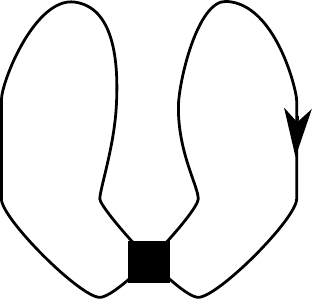} \end{array}
\Bigg \rangle + 2 \Bigg \langle \begin{array}{c}  \scalebox{0.3}{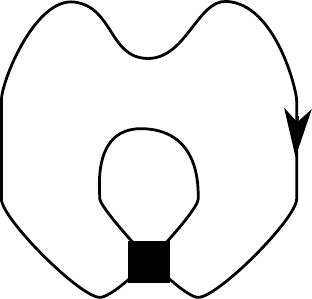} \end{array}
\Bigg \rangle + A^{-2} \Bigg \langle  \begin{array}{c} \scalebox{0.3}{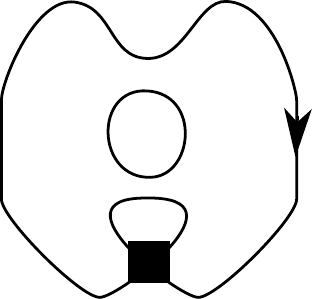} \end{array}
\Bigg \rangle \\
&= A^2 \left( V \Bigg \langle  \begin{array}{c} \scalebox{0.3}{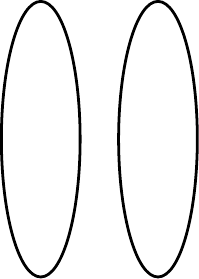} \end{array}
\Bigg \rangle + H \Bigg \langle \begin{array}{c} \scalebox{0.3}{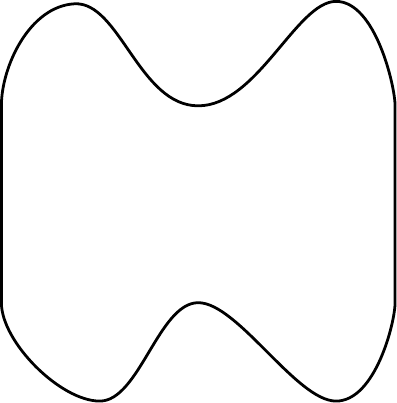}
\end{array} \Bigg \rangle \right)
+ 2 V \Bigg \langle
\begin{array}{c} \scalebox{0.3}{\input{diagrams/ptrefsingle.pdf_tex}} \end{array} \Bigg \rangle 
 \\
&  + 2 H \Bigg \langle \begin{array}{c} \scalebox{0.3}{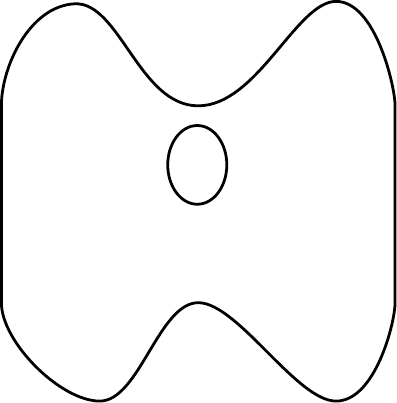} 
\end{array} \Bigg \rangle  
+A^{-2} \left( 
V \Bigg \langle  \begin{array}{c}  \scalebox{0.3}{\input{diagrams/ptrefdouble.pdf_tex}} \end{array}
\Bigg \rangle + H \Bigg \langle \begin{array}{c}  \scalebox{0.3}{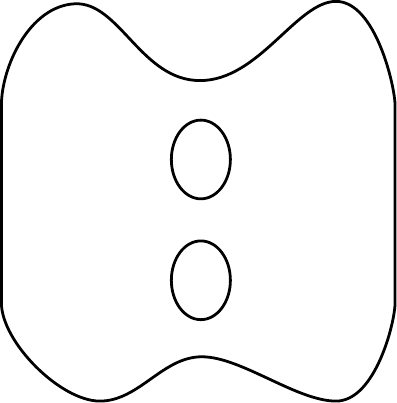} \end{array}
\Bigg \rangle  \right) \\
&=-A^{-8} V + A^{-6} -A^4 V 
\end{align*} 
Then, $P_{PT} (A,V) =A^{-12} + V A^{-14} - V A^{-2} $.
\end{example} 
\section{An obstruction}
Let $D_{+}$ be a classical knot diagram with a selected positive crossing.
The classical knot diagram $D_{-}$ is obtained from $D_{+}$ by changing the selected crossing to a negative crossing. 
The pseudo diagram $D_{\blacksquare} $ is obtained from $D_{+}$ by changing the selected crossing to a pseudo crossing. 
Suppose that the selected crossing is cosmetic and $D_{+} \sim D_{-}$. Then we obtain the following theorem. 

\begin{theorem} For all  knot diagrams $D_{+}$ with a cosmetic crossing $c$, 
$ \langle D_{ \blacksquare} \rangle $ divides $ \langle D_{+} \rangle $ and $ \langle D_{-}  \rangle $.  \end{theorem}
\begin{proof}
Suppose $D_{+} \sim D_{-}$ and that $D_{+} $ and $ D_{-}$ are related by a single crossing change.
Let $w(D_{+})=w +1 $ then $w-1 = w(D_{-})$. 
Let $K$ be a diagram equivalent to $D_{+}$ with writhe $w$.
 
We use $ G $ to denote $ \langle K \rangle$.
 The normalized $f$-polynomials of $K, D_{+}$, and $D_{-}$ are equivalent:
\begin{align}\label{normpolys}
G(-A^{-3})^w &=  \langle D_{+} \rangle (-A^{-3})^{w+1} & G(-A^{-3})^w&= \langle D_{-} \rangle (-A^{-3})^{w-1}.
\end{align}
 Reducing Equation \ref{normpolys},
\begin{align*}
G&=  \langle D_{+} \rangle (-A^{-3}) & G&= \langle D_{-} \rangle (-A^{3}) .
\end{align*}
 We conclude that 
\begin{align} \label{conpolys}
\langle D_{+} \rangle &= G (-A^3) & \langle D_{-} \rangle &= G (-A^{-3}) .
\end{align}

Partially expand the diagrams $D_+$ and $D_-$ at the selected crossing.
\begin{align} \label{partialexpansion1} 
\langle D_{+} \rangle &= A \langle K_v \rangle  + A^{-1} \langle K_H \rangle ,  \\
\langle D_{-} \rangle &= A^{-1} \langle K_v \rangle + A \langle K_H \rangle . \label{partialexpansion2}
\end{align} 
Then  substitute Equation \ref{conpolys} into Equations \ref{partialexpansion1} and \ref{partialexpansion2}.
\begin{align} \label{rewrite1} 
G (-A^3) &= A \langle K_v \rangle  + A^{-1} \langle K_H \rangle,  \\
G(-A^{-3})  &= A^{-1} \langle K_v \rangle + A \langle K_H \rangle . \label{rewrite2} 
\end{align} 
Multiplying through Equations \ref{rewrite1} and \ref{rewrite2}:
\begin{align} \nonumber
G (-A^2) &= \langle K_v \rangle  + A^{-2} \langle K_H \rangle , \\ \nonumber
G(-A^{-2}) &=  \langle K_v \rangle + A^2 \langle K_H \rangle .
\end{align}

Eliminate $K_V$ from the system of equations:
\begin{align} \label{KHequation}
G(-A^2 + A^{-2}) &= (A^{-2} - A^2) \langle K_H \rangle. 
\end{align}
Reducing Equation \ref{KHequation}, we obtain
\begin{equation} \label{khval}
G = \langle K_H \rangle .
\end{equation}
Using the fact that $ \langle D_{+} \rangle = G (-A^3)$ and Equation \ref{khval},
\begin{align*} 
G (-A^3) &= A \langle K_V \rangle + A^{-1} \langle K_H \rangle \\
G (-A^3) &= A \langle K_V \rangle + A^{-1} G \\
G (-A^3 -A^{-1}) &= A \langle K_V \rangle  \\
G  d &= \langle K_V \rangle.
\end{align*}
As a result, 
\begin{equation*}
\langle K_V \rangle = d G  \text{ and } \langle K_H \rangle = G .
\end{equation*}
Apply the result to the expansion of $ D_{\blacksquare}$:
\begin{align*}
\langle D_{\blacksquare} \rangle &= V \langle K_V \rangle + (1 - V d) \langle K_H \rangle \\
&= V d G + (1- V d) G \\
&= G .
\end{align*} 

Then $ \langle D_{\blacksquare} \rangle $ divides  
$ \langle D_{+} \rangle $ and $ \langle D_{-}  \rangle $. 
\end{proof} 

We obtain the following corollary.

\begin{corollary}For all  knot diagrams $D$ with a cosmetic crossing $c$, 
$ \langle D_{ \blacksquare} \rangle $ has no summands with a power of $V$.  \end{corollary} 
\begin{example}
The obstruction is demonstrated using the trefoil knot. By symmetry, none of the crossings are cosmetic. 
\begin{align*} 
\Bigg \langle 
\begin{array}{c} \scalebox{0.3}{\input{diagrams/pseudotrefoil.pdf_tex}}
\end{array} \Bigg \rangle  &=-A^{-8} V + A^{-6} -A^4 V  . \\
\Bigg \langle  \begin{array}{c} \scalebox{0.3}{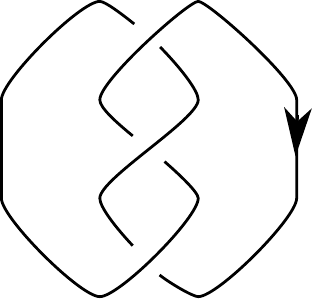}
\end{array} \Bigg \rangle &= A^{-7} -A^{-3} - A^5 .
\end{align*}
\end{example}
\begin{example} 
We consider the first non-alternating classical knot, $K_{11n1}$ \cite{knotinfo}, shown in Figure \ref{fig:knot11n1}. In Figure \ref{fig:knot11n1f}, we select a crossing to construct 
$ K_{11n1 \blacksquare} $. 
\begin{figure}
\begin{subfigure}{0.48\linewidth}
\[  \begin{array}{c}
\scalebox{0.5}{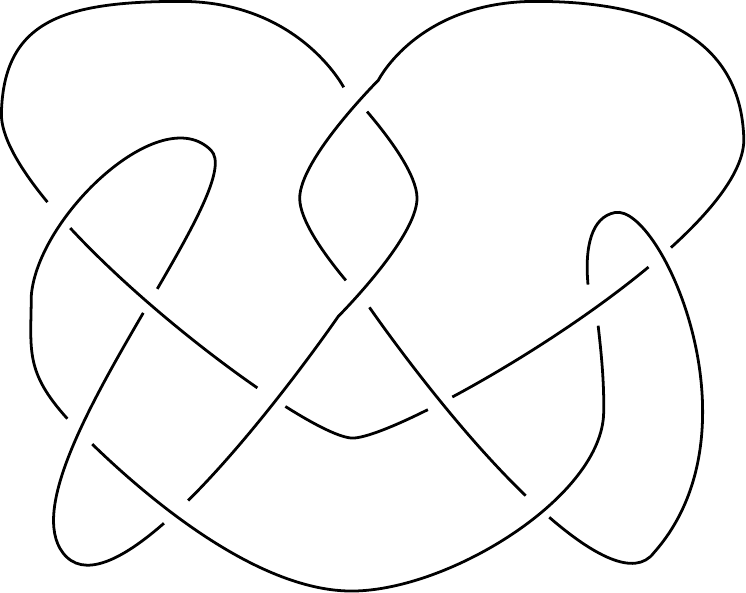} \end{array} \]
\caption{$K_{11n1}$}
\label{fig:knot11n1}
\end{subfigure}
\begin{subfigure}{0.48\linewidth}
\[  \begin{array}{c}
\scalebox{0.5}{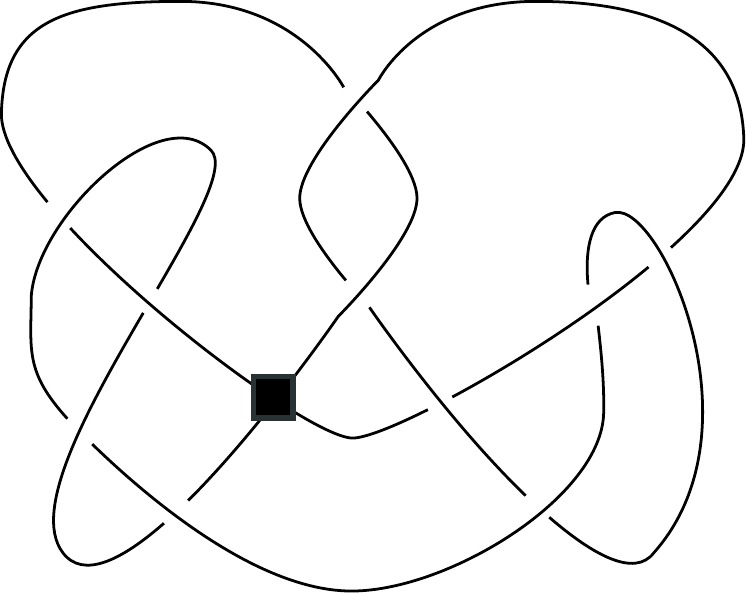} \end{array} \]
\caption{$K_{11n1 \blacksquare}$}
\label{fig:knot11n1f}
\end{subfigure}
\caption{Example 2}
\label{fig:example2}
\end{figure}

\begin{equation}
\langle K_{11n1} \rangle = 
- \frac{(1 + 2 A^8 - A^{12} + A^{28} - A^{32} + A^{36} - A^{40})}{(A^{17} (1 + A^4))}
\end{equation} 

\begin{align} 
\langle K_{11n1 \blacksquare} \rangle &=A^{-24} (A^2 - 3 A^6 + 5 A^{10} - 7 A^{14} + 9 A^{18} - 9 A^{22} + 8 A^{26}   - 6 A^{30}  \\ \nonumber & + 4 A^{34} - 2 A^{38}  
 + A^{42} + V - 3 A^4 V + 4 A^8 V - 
   6 A^{12} V \\ \nonumber
  & + 6 A^{16} V - 5 A^{20} V + 4 A^{24} V - 2 A^{28} V  
 + A^{36} V -   
   A^{40} V + A^{44} V)
\end{align}

This calculation determines that the selected crossing in $K_{11n1}$ is not cosmetic.
\end{example} 

\section{Conclusion} 
The pseudo bracket polynomial is not only an invariant of psuedo knots, but also provides a computable obstruction to cosmetic crossings.

\bibliographystyle{plain}
\bibliography{cosmetic}

\end{document}

%% file: diagrams/r1moverhs.pdf_tex
\begingroup%
  \makeatletter%
  \providecommand\color[2][]{%
    \errmessage{(Inkscape) Color is used for the text in Inkscape, but the package 'color.sty' is not loaded}%
    \renewcommand\color[2][]{}%
  }%
  \providecommand\transparent[1]{%
    \errmessage{(Inkscape) Transparency is used (non-zero) for the text in Inkscape, but the package 'transparent.sty' is not loaded}%
    \renewcommand\transparent[1]{}%
  }%
  \providecommand\rotatebox[2]{#2}%
  \ifx\svgwidth\undefined%
    \setlength{\unitlength}{69.55803981bp}%
    \ifx\svgscale\undefined%
      \relax%
    \else%
      \setlength{\unitlength}{\unitlength * \real{\svgscale}}%
    \fi%
  \else%
    \setlength{\unitlength}{\svgwidth}%
  \fi%
  \global\let\svgwidth\undefined%
  \global\let\svgscale\undefined%
  \makeatother%
  \begin{picture}(1,0.8373183)%
    \put(0,0){\includegraphics[width=\unitlength]{r1moverhs.pdf}}%
  \end{picture}%
\endgroup%

%% file: diagrams/r2movelhs.pdf_tex
\begingroup%
  \makeatletter%
  \providecommand\color[2][]{%
    \errmessage{(Inkscape) Color is used for the text in Inkscape, but the package 'color.sty' is not loaded}%
    \renewcommand\color[2][]{}%
  }%
  \providecommand\transparent[1]{%
    \errmessage{(Inkscape) Transparency is used (non-zero) for the text in Inkscape, but the package 'transparent.sty' is not loaded}%
    \renewcommand\transparent[1]{}%
  }%
  \providecommand\rotatebox[2]{#2}%
  \ifx\svgwidth\undefined%
    \setlength{\unitlength}{88.180302bp}%
    \ifx\svgscale\undefined%
      \relax%
    \else%
      \setlength{\unitlength}{\unitlength * \real{\svgscale}}%
    \fi%
  \else%
    \setlength{\unitlength}{\svgwidth}%
  \fi%
  \global\let\svgwidth\undefined%
  \global\let\svgscale\undefined%
  \makeatother%
  \begin{picture}(1,0.987534)%
    \put(0,0){\includegraphics[width=\unitlength,page=1]{r2movelhs.pdf}}%
  \end{picture}%
\endgroup%

%% file: diagrams/r2moverhs.pdf_tex
\begingroup%
  \makeatletter%
  \providecommand\color[2][]{%
    \errmessage{(Inkscape) Color is used for the text in Inkscape, but the package 'color.sty' is not loaded}%
    \renewcommand\color[2][]{}%
  }%
  \providecommand\transparent[1]{%
    \errmessage{(Inkscape) Transparency is used (non-zero) for the text in Inkscape, but the package 'transparent.sty' is not loaded}%
    \renewcommand\transparent[1]{}%
  }%
  \providecommand\rotatebox[2]{#2}%
  \ifx\svgwidth\undefined%
    \setlength{\unitlength}{87.789499bp}%
    \ifx\svgscale\undefined%
      \relax%
    \else%
      \setlength{\unitlength}{\unitlength * \real{\svgscale}}%
    \fi%
  \else%
    \setlength{\unitlength}{\svgwidth}%
  \fi%
  \global\let\svgwidth\undefined%
  \global\let\svgscale\undefined%
  \makeatother%
  \begin{picture}(1,0.98085195)%
    \put(0,0){\includegraphics[width=\unitlength,page=1]{r2moverhs.pdf}}%
  \end{picture}%
\endgroup%

%% file: diagrams/r3movelhs.pdf_tex
\begingroup%
  \makeatletter%
  \providecommand\color[2][]{%
    \errmessage{(Inkscape) Color is used for the text in Inkscape, but the package 'color.sty' is not loaded}%
    \renewcommand\color[2][]{}%
  }%
  \providecommand\transparent[1]{%
    \errmessage{(Inkscape) Transparency is used (non-zero) for the text in Inkscape, but the package 'transparent.sty' is not loaded}%
    \renewcommand\transparent[1]{}%
  }%
  \providecommand\rotatebox[2]{#2}%
  \ifx\svgwidth\undefined%
    \setlength{\unitlength}{142.59632729bp}%
    \ifx\svgscale\undefined%
      \relax%
    \else%
      \setlength{\unitlength}{\unitlength * \real{\svgscale}}%
    \fi%
  \else%
    \setlength{\unitlength}{\svgwidth}%
  \fi%
  \global\let\svgwidth\undefined%
  \global\let\svgscale\undefined%
  \makeatother%
  \begin{picture}(1,0.81333476)%
    \put(0,0){\includegraphics[width=\unitlength]{r3movelhs.pdf}}%
  \end{picture}%
\endgroup%

%% file: diagrams/r3moverhs.pdf_tex
\begingroup%
  \makeatletter%
  \providecommand\color[2][]{%
    \errmessage{(Inkscape) Color is used for the text in Inkscape, but the package 'color.sty' is not loaded}%
    \renewcommand\color[2][]{}%
  }%
  \providecommand\transparent[1]{%
    \errmessage{(Inkscape) Transparency is used (non-zero) for the text in Inkscape, but the package 'transparent.sty' is not loaded}%
    \renewcommand\transparent[1]{}%
  }%
  \providecommand\rotatebox[2]{#2}%
  \ifx\svgwidth\undefined%
    \setlength{\unitlength}{142.59154429bp}%
    \ifx\svgscale\undefined%
      \relax%
    \else%
      \setlength{\unitlength}{\unitlength * \real{\svgscale}}%
    \fi%
  \else%
    \setlength{\unitlength}{\svgwidth}%
  \fi%
  \global\let\svgwidth\undefined%
  \global\let\svgscale\undefined%
  \makeatother%
  \begin{picture}(1,0.81335312)%
    \put(0,0){\includegraphics[width=\unitlength,page=1]{r3moverhs.pdf}}%
  \end{picture}%
\endgroup%

%% file: diagrams/p2movelhs.pdf_tex
\begingroup%
  \makeatletter%
  \providecommand\color[2][]{%
    \errmessage{(Inkscape) Color is used for the text in Inkscape, but the package 'color.sty' is not loaded}%
    \renewcommand\color[2][]{}%
  }%
  \providecommand\transparent[1]{%
    \errmessage{(Inkscape) Transparency is used (non-zero) for the text in Inkscape, but the package 'transparent.sty' is not loaded}%
    \renewcommand\transparent[1]{}%
  }%
  \providecommand\rotatebox[2]{#2}%
  \ifx\svgwidth\undefined%
    \setlength{\unitlength}{87.78747013bp}%
    \ifx\svgscale\undefined%
      \relax%
    \else%
      \setlength{\unitlength}{\unitlength * \real{\svgscale}}%
    \fi%
  \else%
    \setlength{\unitlength}{\svgwidth}%
  \fi%
  \global\let\svgwidth\undefined%
  \global\let\svgscale\undefined%
  \makeatother%
  \begin{picture}(1,0.98087462)%
    \put(0,0){\includegraphics[width=\unitlength,page=1]{p2movelhs.pdf}}%
  \end{picture}%
\endgroup%

%% file: diagrams/p2moverhs.pdf_tex
\begingroup%
  \makeatletter%
  \providecommand\color[2][]{%
    \errmessage{(Inkscape) Color is used for the text in Inkscape, but the package 'color.sty' is not loaded}%
    \renewcommand\color[2][]{}%
  }%
  \providecommand\transparent[1]{%
    \errmessage{(Inkscape) Transparency is used (non-zero) for the text in Inkscape, but the package 'transparent.sty' is not loaded}%
    \renewcommand\transparent[1]{}%
  }%
  \providecommand\rotatebox[2]{#2}%
  \ifx\svgwidth\undefined%
    \setlength{\unitlength}{93.46351762bp}%
    \ifx\svgscale\undefined%
      \relax%
    \else%
      \setlength{\unitlength}{\unitlength * \real{\svgscale}}%
    \fi%
  \else%
    \setlength{\unitlength}{\svgwidth}%
  \fi%
  \global\let\svgwidth\undefined%
  \global\let\svgscale\undefined%
  \makeatother%
  \begin{picture}(1,0.94990194)%
    \put(0,0){\includegraphics[width=\unitlength,page=1]{p2moverhs.pdf}}%
  \end{picture}%
\endgroup%

%% file: diagrams/p3movelhs.pdf_tex
\begingroup%
  \makeatletter%
  \providecommand\color[2][]{%
    \errmessage{(Inkscape) Color is used for the text in Inkscape, but the package 'color.sty' is not loaded}%
    \renewcommand\color[2][]{}%
  }%
  \providecommand\transparent[1]{%
    \errmessage{(Inkscape) Transparency is used (non-zero) for the text in Inkscape, but the package 'transparent.sty' is not loaded}%
    \renewcommand\transparent[1]{}%
  }%
  \providecommand\rotatebox[2]{#2}%
  \ifx\svgwidth\undefined%
    \setlength{\unitlength}{142.59154349bp}%
    \ifx\svgscale\undefined%
      \relax%
    \else%
      \setlength{\unitlength}{\unitlength * \real{\svgscale}}%
    \fi%
  \else%
    \setlength{\unitlength}{\svgwidth}%
  \fi%
  \global\let\svgwidth\undefined%
  \global\let\svgscale\undefined%
  \makeatother%
  \begin{picture}(1,0.81335274)%
    \put(0,0){\includegraphics[width=\unitlength,page=1]{p3movelhs.pdf}}%
  \end{picture}%
\endgroup%

%% file: diagrams/p3moverhs.pdf_tex
\begingroup%
  \makeatletter%
  \providecommand\color[2][]{%
    \errmessage{(Inkscape) Color is used for the text in Inkscape, but the package 'color.sty' is not loaded}%
    \renewcommand\color[2][]{}%
  }%
  \providecommand\transparent[1]{%
    \errmessage{(Inkscape) Transparency is used (non-zero) for the text in Inkscape, but the package 'transparent.sty' is not loaded}%
    \renewcommand\transparent[1]{}%
  }%
  \providecommand\rotatebox[2]{#2}%
  \ifx\svgwidth\undefined%
    \setlength{\unitlength}{142.59154429bp}%
    \ifx\svgscale\undefined%
      \relax%
    \else%
      \setlength{\unitlength}{\unitlength * \real{\svgscale}}%
    \fi%
  \else%
    \setlength{\unitlength}{\svgwidth}%
  \fi%
  \global\let\svgwidth\undefined%
  \global\let\svgscale\undefined%
  \makeatother%
  \begin{picture}(1,0.81335312)%
    \put(0,0){\includegraphics[width=\unitlength,page=1]{p3moverhs.pdf}}%
  \end{picture}%
\endgroup%

%% file: diagrams/positivecrossing.pdf_tex
\begingroup%
  \makeatletter%
  \providecommand\color[2][]{%
    \errmessage{(Inkscape) Color is used for the text in Inkscape, but the package 'color.sty' is not loaded}%
    \renewcommand\color[2][]{}%
  }%
  \providecommand\transparent[1]{%
    \errmessage{(Inkscape) Transparency is used (non-zero) for the text in Inkscape, but the package 'transparent.sty' is not loaded}%
    \renewcommand\transparent[1]{}%
  }%
  \providecommand\rotatebox[2]{#2}%
  \ifx\svgwidth\undefined%
    \setlength{\unitlength}{58.88773498bp}%
    \ifx\svgscale\undefined%
      \relax%
    \else%
      \setlength{\unitlength}{\unitlength * \real{\svgscale}}%
    \fi%
  \else%
    \setlength{\unitlength}{\svgwidth}%
  \fi%
  \global\let\svgwidth\undefined%
  \global\let\svgscale\undefined%
  \makeatother%
  \begin{picture}(1,0.98903821)%
    \put(0,0){\includegraphics[width=\unitlength]{positivecrossing.pdf}}%
  \end{picture}%
\endgroup%

%% file: diagrams/vsmooth.pdf_tex
\begingroup%
  \makeatletter%
  \providecommand\color[2][]{%
    \errmessage{(Inkscape) Color is used for the text in Inkscape, but the package 'color.sty' is not loaded}%
    \renewcommand\color[2][]{}%
  }%
  \providecommand\transparent[1]{%
    \errmessage{(Inkscape) Transparency is used (non-zero) for the text in Inkscape, but the package 'transparent.sty' is not loaded}%
    \renewcommand\transparent[1]{}%
  }%
  \providecommand\rotatebox[2]{#2}%
  \ifx\svgwidth\undefined%
    \setlength{\unitlength}{58.93568977bp}%
    \ifx\svgscale\undefined%
      \relax%
    \else%
      \setlength{\unitlength}{\unitlength * \real{\svgscale}}%
    \fi%
  \else%
    \setlength{\unitlength}{\svgwidth}%
  \fi%
  \global\let\svgwidth\undefined%
  \global\let\svgscale\undefined%
  \makeatother%
  \begin{picture}(1,0.98874533)%
    \put(0,0){\includegraphics[width=\unitlength,page=1]{vsmooth.pdf}}%
  \end{picture}%
\endgroup%

%% file: diagrams/hsmooth.pdf_tex
\begingroup%
  \makeatletter%
  \providecommand\color[2][]{%
    \errmessage{(Inkscape) Color is used for the text in Inkscape, but the package 'color.sty' is not loaded}%
    \renewcommand\color[2][]{}%
  }%
  \providecommand\transparent[1]{%
    \errmessage{(Inkscape) Transparency is used (non-zero) for the text in Inkscape, but the package 'transparent.sty' is not loaded}%
    \renewcommand\transparent[1]{}%
  }%
  \providecommand\rotatebox[2]{#2}%
  \ifx\svgwidth\undefined%
    \setlength{\unitlength}{57.54968285bp}%
    \ifx\svgscale\undefined%
      \relax%
    \else%
      \setlength{\unitlength}{\unitlength * \real{\svgscale}}%
    \fi%
  \else%
    \setlength{\unitlength}{\svgwidth}%
  \fi%
  \global\let\svgwidth\undefined%
  \global\let\svgscale\undefined%
  \makeatother%
  \begin{picture}(1,1.00890175)%
    \put(0,0){\includegraphics[width=\unitlength,page=1]{hsmooth.pdf}}%
  \end{picture}%
\endgroup%

%% file: diagrams/negativecrossing.pdf_tex
\begingroup%
  \makeatletter%
  \providecommand\color[2][]{%
    \errmessage{(Inkscape) Color is used for the text in Inkscape, but the package 'color.sty' is not loaded}%
    \renewcommand\color[2][]{}%
  }%
  \providecommand\transparent[1]{%
    \errmessage{(Inkscape) Transparency is used (non-zero) for the text in Inkscape, but the package 'transparent.sty' is not loaded}%
    \renewcommand\transparent[1]{}%
  }%
  \providecommand\rotatebox[2]{#2}%
  \ifx\svgwidth\undefined%
    \setlength{\unitlength}{58.88773498bp}%
    \ifx\svgscale\undefined%
      \relax%
    \else%
      \setlength{\unitlength}{\unitlength * \real{\svgscale}}%
    \fi%
  \else%
    \setlength{\unitlength}{\svgwidth}%
  \fi%
  \global\let\svgwidth\undefined%
  \global\let\svgscale\undefined%
  \makeatother%
  \begin{picture}(1,0.98903821)%
    \put(0,0){\includegraphics[width=\unitlength]{negativecrossing.pdf}}%
  \end{picture}%
\endgroup%

%% file: diagrams/psuedocrossing.pdf_tex
\begingroup%
  \makeatletter%
  \providecommand\color[2][]{%
    \errmessage{(Inkscape) Color is used for the text in Inkscape, but the package 'color.sty' is not loaded}%
    \renewcommand\color[2][]{}%
  }%
  \providecommand\transparent[1]{%
    \errmessage{(Inkscape) Transparency is used (non-zero) for the text in Inkscape, but the package 'transparent.sty' is not loaded}%
    \renewcommand\transparent[1]{}%
  }%
  \providecommand\rotatebox[2]{#2}%
  \ifx\svgwidth\undefined%
    \setlength{\unitlength}{58.88773498bp}%
    \ifx\svgscale\undefined%
      \relax%
    \else%
      \setlength{\unitlength}{\unitlength * \real{\svgscale}}%
    \fi%
  \else%
    \setlength{\unitlength}{\svgwidth}%
  \fi%
  \global\let\svgwidth\undefined%
  \global\let\svgscale\undefined%
  \makeatother%
  \begin{picture}(1,0.98903821)%
    \put(0,0){\includegraphics[width=\unitlength]{psuedocrossing.pdf}}%
  \end{picture}%
\endgroup%

%% file: diagrams/pseudor3lhs.pdf_tex
\begingroup%
  \makeatletter%
  \providecommand\color[2][]{%
    \errmessage{(Inkscape) Color is used for the text in Inkscape, but the package 'color.sty' is not loaded}%
    \renewcommand\color[2][]{}%
  }%
  \providecommand\transparent[1]{%
    \errmessage{(Inkscape) Transparency is used (non-zero) for the text in Inkscape, but the package 'transparent.sty' is not loaded}%
    \renewcommand\transparent[1]{}%
  }%
  \providecommand\rotatebox[2]{#2}%
  \ifx\svgwidth\undefined%
    \setlength{\unitlength}{80.94912662bp}%
    \ifx\svgscale\undefined%
      \relax%
    \else%
      \setlength{\unitlength}{\unitlength * \real{\svgscale}}%
    \fi%
  \else%
    \setlength{\unitlength}{\svgwidth}%
  \fi%
  \global\let\svgwidth\undefined%
  \global\let\svgscale\undefined%
  \makeatother%
  \begin{picture}(1,1.08723909)%
    \put(0,0){\includegraphics[width=\unitlength,page=1]{pseudor3lhs.pdf}}%
  \end{picture}%
\endgroup%

%% file: diagrams/pseudor3vsmooth.pdf_tex
\begingroup%
  \makeatletter%
  \providecommand\color[2][]{%
    \errmessage{(Inkscape) Color is used for the text in Inkscape, but the package 'color.sty' is not loaded}%
    \renewcommand\color[2][]{}%
  }%
  \providecommand\transparent[1]{%
    \errmessage{(Inkscape) Transparency is used (non-zero) for the text in Inkscape, but the package 'transparent.sty' is not loaded}%
    \renewcommand\transparent[1]{}%
  }%
  \providecommand\rotatebox[2]{#2}%
  \ifx\svgwidth\undefined%
    \setlength{\unitlength}{80.90751129bp}%
    \ifx\svgscale\undefined%
      \relax%
    \else%
      \setlength{\unitlength}{\unitlength * \real{\svgscale}}%
    \fi%
  \else%
    \setlength{\unitlength}{\svgwidth}%
  \fi%
  \global\let\svgwidth\undefined%
  \global\let\svgscale\undefined%
  \makeatother%
  \begin{picture}(1,1.0772218)%
    \put(0,0){\includegraphics[width=\unitlength,page=1]{pseudor3vsmooth.pdf}}%
  \end{picture}%
\endgroup%

%% file: diagrams/pseudor3hsmooth.pdf_tex
\begingroup%
  \makeatletter%
  \providecommand\color[2][]{%
    \errmessage{(Inkscape) Color is used for the text in Inkscape, but the package 'color.sty' is not loaded}%
    \renewcommand\color[2][]{}%
  }%
  \providecommand\transparent[1]{%
    \errmessage{(Inkscape) Transparency is used (non-zero) for the text in Inkscape, but the package 'transparent.sty' is not loaded}%
    \renewcommand\transparent[1]{}%
  }%
  \providecommand\rotatebox[2]{#2}%
  \ifx\svgwidth\undefined%
    \setlength{\unitlength}{80.90751129bp}%
    \ifx\svgscale\undefined%
      \relax%
    \else%
      \setlength{\unitlength}{\unitlength * \real{\svgscale}}%
    \fi%
  \else%
    \setlength{\unitlength}{\svgwidth}%
  \fi%
  \global\let\svgwidth\undefined%
  \global\let\svgscale\undefined%
  \makeatother%
  \begin{picture}(1,1.07723234)%
    \put(0,0){\includegraphics[width=\unitlength,page=1]{pseudor3hsmooth.pdf}}%
  \end{picture}%
\endgroup%

%% file: diagrams/pseudor3hsmooth2.pdf_tex
\begingroup%
  \makeatletter%
  \providecommand\color[2][]{%
    \errmessage{(Inkscape) Color is used for the text in Inkscape, but the package 'color.sty' is not loaded}%
    \renewcommand\color[2][]{}%
  }%
  \providecommand\transparent[1]{%
    \errmessage{(Inkscape) Transparency is used (non-zero) for the text in Inkscape, but the package 'transparent.sty' is not loaded}%
    \renewcommand\transparent[1]{}%
  }%
  \providecommand\rotatebox[2]{#2}%
  \ifx\svgwidth\undefined%
    \setlength{\unitlength}{80.90751129bp}%
    \ifx\svgscale\undefined%
      \relax%
    \else%
      \setlength{\unitlength}{\unitlength * \real{\svgscale}}%
    \fi%
  \else%
    \setlength{\unitlength}{\svgwidth}%
  \fi%
  \global\let\svgwidth\undefined%
  \global\let\svgscale\undefined%
  \makeatother%
  \begin{picture}(1,1.07723237)%
    \put(0,0){\includegraphics[width=\unitlength,page=1]{pseudor3hsmooth2.pdf}}%
  \end{picture}%
\endgroup%

%% file: diagrams/pseudor3rhs.pdf_tex
\begingroup%
  \makeatletter%
  \providecommand\color[2][]{%
    \errmessage{(Inkscape) Color is used for the text in Inkscape, but the package 'color.sty' is not loaded}%
    \renewcommand\color[2][]{}%
  }%
  \providecommand\transparent[1]{%
    \errmessage{(Inkscape) Transparency is used (non-zero) for the text in Inkscape, but the package 'transparent.sty' is not loaded}%
    \renewcommand\transparent[1]{}%
  }%
  \providecommand\rotatebox[2]{#2}%
  \ifx\svgwidth\undefined%
    \setlength{\unitlength}{80.90751129bp}%
    \ifx\svgscale\undefined%
      \relax%
    \else%
      \setlength{\unitlength}{\unitlength * \real{\svgscale}}%
    \fi%
  \else%
    \setlength{\unitlength}{\svgwidth}%
  \fi%
  \global\let\svgwidth\undefined%
  \global\let\svgscale\undefined%
  \makeatother%
  \begin{picture}(1,1.08779832)%
    \put(0,0){\includegraphics[width=\unitlength,page=1]{pseudor3rhs.pdf}}%
  \end{picture}%
\endgroup%

%% file: diagrams/psuedo2movelhs.pdf_tex
\begingroup%
  \makeatletter%
  \providecommand\color[2][]{%
    \errmessage{(Inkscape) Color is used for the text in Inkscape, but the package 'color.sty' is not loaded}%
    \renewcommand\color[2][]{}%
  }%
  \providecommand\transparent[1]{%
    \errmessage{(Inkscape) Transparency is used (non-zero) for the text in Inkscape, but the package 'transparent.sty' is not loaded}%
    \renewcommand\transparent[1]{}%
  }%
  \providecommand\rotatebox[2]{#2}%
  \ifx\svgwidth\undefined%
    \setlength{\unitlength}{86.48308778bp}%
    \ifx\svgscale\undefined%
      \relax%
    \else%
      \setlength{\unitlength}{\unitlength * \real{\svgscale}}%
    \fi%
  \else%
    \setlength{\unitlength}{\svgwidth}%
  \fi%
  \global\let\svgwidth\undefined%
  \global\let\svgscale\undefined%
  \makeatother%
  \begin{picture}(1,1.66686861)%
    \put(0,0){\includegraphics[width=\unitlength,page=1]{psuedo2movelhs.pdf}}%
  \end{picture}%
\endgroup%

%% file: diagrams/psuedo2movevsmooth1.pdf_tex
\begingroup%
  \makeatletter%
  \providecommand\color[2][]{%
    \errmessage{(Inkscape) Color is used for the text in Inkscape, but the package 'color.sty' is not loaded}%
    \renewcommand\color[2][]{}%
  }%
  \providecommand\transparent[1]{%
    \errmessage{(Inkscape) Transparency is used (non-zero) for the text in Inkscape, but the package 'transparent.sty' is not loaded}%
    \renewcommand\transparent[1]{}%
  }%
  \providecommand\rotatebox[2]{#2}%
  \ifx\svgwidth\undefined%
    \setlength{\unitlength}{84.77370387bp}%
    \ifx\svgscale\undefined%
      \relax%
    \else%
      \setlength{\unitlength}{\unitlength * \real{\svgscale}}%
    \fi%
  \else%
    \setlength{\unitlength}{\svgwidth}%
  \fi%
  \global\let\svgwidth\undefined%
  \global\let\svgscale\undefined%
  \makeatother%
  \begin{picture}(1,1.69324202)%
    \put(0,0){\includegraphics[width=\unitlength,page=1]{psuedo2movevsmooth1.pdf}}%
  \end{picture}%
\endgroup%

%% file: diagrams/psuedo2movehsmooth1.pdf_tex
\begingroup%
  \makeatletter%
  \providecommand\color[2][]{%
    \errmessage{(Inkscape) Color is used for the text in Inkscape, but the package 'color.sty' is not loaded}%
    \renewcommand\color[2][]{}%
  }%
  \providecommand\transparent[1]{%
    \errmessage{(Inkscape) Transparency is used (non-zero) for the text in Inkscape, but the package 'transparent.sty' is not loaded}%
    \renewcommand\transparent[1]{}%
  }%
  \providecommand\rotatebox[2]{#2}%
  \ifx\svgwidth\undefined%
    \setlength{\unitlength}{84.77191427bp}%
    \ifx\svgscale\undefined%
      \relax%
    \else%
      \setlength{\unitlength}{\unitlength * \real{\svgscale}}%
    \fi%
  \else%
    \setlength{\unitlength}{\svgwidth}%
  \fi%
  \global\let\svgwidth\undefined%
  \global\let\svgscale\undefined%
  \makeatother%
  \begin{picture}(1,1.69329556)%
    \put(0,0){\includegraphics[width=\unitlength,page=1]{psuedo2movehsmooth1.pdf}}%
  \end{picture}%
\endgroup%

%% file: diagrams/psuedo2movevsmooth2.pdf_tex
\begingroup%
  \makeatletter%
  \providecommand\color[2][]{%
    \errmessage{(Inkscape) Color is used for the text in Inkscape, but the package 'color.sty' is not loaded}%
    \renewcommand\color[2][]{}%
  }%
  \providecommand\transparent[1]{%
    \errmessage{(Inkscape) Transparency is used (non-zero) for the text in Inkscape, but the package 'transparent.sty' is not loaded}%
    \renewcommand\transparent[1]{}%
  }%
  \providecommand\rotatebox[2]{#2}%
  \ifx\svgwidth\undefined%
    \setlength{\unitlength}{84.77370719bp}%
    \ifx\svgscale\undefined%
      \relax%
    \else%
      \setlength{\unitlength}{\unitlength * \real{\svgscale}}%
    \fi%
  \else%
    \setlength{\unitlength}{\svgwidth}%
  \fi%
  \global\let\svgwidth\undefined%
  \global\let\svgscale\undefined%
  \makeatother%
  \begin{picture}(1,1.69324195)%
    \put(0,0){\includegraphics[width=\unitlength,page=1]{psuedo2movevsmooth2.pdf}}%
  \end{picture}%
\endgroup%

%% file: diagrams/psuedo2movehsmooth2.pdf_tex
\begingroup%
  \makeatletter%
  \providecommand\color[2][]{%
    \errmessage{(Inkscape) Color is used for the text in Inkscape, but the package 'color.sty' is not loaded}%
    \renewcommand\color[2][]{}%
  }%
  \providecommand\transparent[1]{%
    \errmessage{(Inkscape) Transparency is used (non-zero) for the text in Inkscape, but the package 'transparent.sty' is not loaded}%
    \renewcommand\transparent[1]{}%
  }%
  \providecommand\rotatebox[2]{#2}%
  \ifx\svgwidth\undefined%
    \setlength{\unitlength}{84.76974541bp}%
    \ifx\svgscale\undefined%
      \relax%
    \else%
      \setlength{\unitlength}{\unitlength * \real{\svgscale}}%
    \fi%
  \else%
    \setlength{\unitlength}{\svgwidth}%
  \fi%
  \global\let\svgwidth\undefined%
  \global\let\svgscale\undefined%
  \makeatother%
  \begin{picture}(1,1.70558406)%
    \put(0,0){\includegraphics[width=\unitlength,page=1]{psuedo2movehsmooth2.pdf}}%
  \end{picture}%
\endgroup%

%% file: diagrams/psuedo2moverhs.pdf_tex
\begingroup%
  \makeatletter%
  \providecommand\color[2][]{%
    \errmessage{(Inkscape) Color is used for the text in Inkscape, but the package 'color.sty' is not loaded}%
    \renewcommand\color[2][]{}%
  }%
  \providecommand\transparent[1]{%
    \errmessage{(Inkscape) Transparency is used (non-zero) for the text in Inkscape, but the package 'transparent.sty' is not loaded}%
    \renewcommand\transparent[1]{}%
  }%
  \providecommand\rotatebox[2]{#2}%
  \ifx\svgwidth\undefined%
    \setlength{\unitlength}{86.48308776bp}%
    \ifx\svgscale\undefined%
      \relax%
    \else%
      \setlength{\unitlength}{\unitlength * \real{\svgscale}}%
    \fi%
  \else%
    \setlength{\unitlength}{\svgwidth}%
  \fi%
  \global\let\svgwidth\undefined%
  \global\let\svgscale\undefined%
  \makeatother%
  \begin{picture}(1,1.66686861)%
    \put(0,0){\includegraphics[width=\unitlength,page=1]{psuedo2moverhs.pdf}}%
  \end{picture}%
\endgroup%

%% file: diagrams/pr1moverhs.pdf_tex
\begingroup%
  \makeatletter%
  \providecommand\color[2][]{%
    \errmessage{(Inkscape) Color is used for the text in Inkscape, but the package 'color.sty' is not loaded}%
    \renewcommand\color[2][]{}%
  }%
  \providecommand\transparent[1]{%
    \errmessage{(Inkscape) Transparency is used (non-zero) for the text in Inkscape, but the package 'transparent.sty' is not loaded}%
    \renewcommand\transparent[1]{}%
  }%
  \providecommand\rotatebox[2]{#2}%
  \ifx\svgwidth\undefined%
    \setlength{\unitlength}{69.55111057bp}%
    \ifx\svgscale\undefined%
      \relax%
    \else%
      \setlength{\unitlength}{\unitlength * \real{\svgscale}}%
    \fi%
  \else%
    \setlength{\unitlength}{\svgwidth}%
  \fi%
  \global\let\svgwidth\undefined%
  \global\let\svgscale\undefined%
  \makeatother%
  \begin{picture}(1,0.83703659)%
    \put(0,0){\includegraphics[width=\unitlength,page=1]{pr1moverhs.pdf}}%
  \end{picture}%
\endgroup%

%% file: diagrams/pseduor1vsmooth.pdf_tex
\begingroup%
  \makeatletter%
  \providecommand\color[2][]{%
    \errmessage{(Inkscape) Color is used for the text in Inkscape, but the package 'color.sty' is not loaded}%
    \renewcommand\color[2][]{}%
  }%
  \providecommand\transparent[1]{%
    \errmessage{(Inkscape) Transparency is used (non-zero) for the text in Inkscape, but the package 'transparent.sty' is not loaded}%
    \renewcommand\transparent[1]{}%
  }%
  \providecommand\rotatebox[2]{#2}%
  \ifx\svgwidth\undefined%
    \setlength{\unitlength}{68.88629532bp}%
    \ifx\svgscale\undefined%
      \relax%
    \else%
      \setlength{\unitlength}{\unitlength * \real{\svgscale}}%
    \fi%
  \else%
    \setlength{\unitlength}{\svgwidth}%
  \fi%
  \global\let\svgwidth\undefined%
  \global\let\svgscale\undefined%
  \makeatother%
  \begin{picture}(1,0.83545977)%
    \put(0,0){\includegraphics[width=\unitlength,page=1]{pseduor1vsmooth.pdf}}%
  \end{picture}%
\endgroup%

%% file: diagrams/pseduor1hsmooth.pdf_tex
\begingroup%
  \makeatletter%
  \providecommand\color[2][]{%
    \errmessage{(Inkscape) Color is used for the text in Inkscape, but the package 'color.sty' is not loaded}%
    \renewcommand\color[2][]{}%
  }%
  \providecommand\transparent[1]{%
    \errmessage{(Inkscape) Transparency is used (non-zero) for the text in Inkscape, but the package 'transparent.sty' is not loaded}%
    \renewcommand\transparent[1]{}%
  }%
  \providecommand\rotatebox[2]{#2}%
  \ifx\svgwidth\undefined%
    \setlength{\unitlength}{69.55111028bp}%
    \ifx\svgscale\undefined%
      \relax%
    \else%
      \setlength{\unitlength}{\unitlength * \real{\svgscale}}%
    \fi%
  \else%
    \setlength{\unitlength}{\svgwidth}%
  \fi%
  \global\let\svgwidth\undefined%
  \global\let\svgscale\undefined%
  \makeatother%
  \begin{picture}(1,0.8370366)%
    \put(0,0){\includegraphics[width=\unitlength,page=1]{pseduor1hsmooth.pdf}}%
  \end{picture}%
\endgroup%

%% file: diagrams/r1movelhs.pdf_tex
\begingroup%
  \makeatletter%
  \providecommand\color[2][]{%
    \errmessage{(Inkscape) Color is used for the text in Inkscape, but the package 'color.sty' is not loaded}%
    \renewcommand\color[2][]{}%
  }%
  \providecommand\transparent[1]{%
    \errmessage{(Inkscape) Transparency is used (non-zero) for the text in Inkscape, but the package 'transparent.sty' is not loaded}%
    \renewcommand\transparent[1]{}%
  }%
  \providecommand\rotatebox[2]{#2}%
  \ifx\svgwidth\undefined%
    \setlength{\unitlength}{52.98709431bp}%
    \ifx\svgscale\undefined%
      \relax%
    \else%
      \setlength{\unitlength}{\unitlength * \real{\svgscale}}%
    \fi%
  \else%
    \setlength{\unitlength}{\svgwidth}%
  \fi%
  \global\let\svgwidth\undefined%
  \global\let\svgscale\undefined%
  \makeatother%
  \begin{picture}(1,1.13011415)%
    \put(0,0){\includegraphics[width=\unitlength,page=1]{r1movelhs.pdf}}%
  \end{picture}%
\endgroup%

%% file: diagrams/pseudotrefoil.pdf_tex
\begingroup%
  \makeatletter%
  \providecommand\color[2][]{%
    \errmessage{(Inkscape) Color is used for the text in Inkscape, but the package 'color.sty' is not loaded}%
    \renewcommand\color[2][]{}%
  }%
  \providecommand\transparent[1]{%
    \errmessage{(Inkscape) Transparency is used (non-zero) for the text in Inkscape, but the package 'transparent.sty' is not loaded}%
    \renewcommand\transparent[1]{}%
  }%
  \providecommand\rotatebox[2]{#2}%
  \ifx\svgwidth\undefined%
    \setlength{\unitlength}{89.80767717bp}%
    \ifx\svgscale\undefined%
      \relax%
    \else%
      \setlength{\unitlength}{\unitlength * \real{\svgscale}}%
    \fi%
  \else%
    \setlength{\unitlength}{\svgwidth}%
  \fi%
  \global\let\svgwidth\undefined%
  \global\let\svgscale\undefined%
  \makeatother%
  \begin{picture}(1,0.95581325)%
    \put(0,0){\includegraphics[width=\unitlength,page=1]{pseudotrefoil.pdf}}%
  \end{picture}%
\endgroup%

%% file: diagrams/pseudotrefoilsmoothv.pdf_tex
\begingroup%
  \makeatletter%
  \providecommand\color[2][]{%
    \errmessage{(Inkscape) Color is used for the text in Inkscape, but the package 'color.sty' is not loaded}%
    \renewcommand\color[2][]{}%
  }%
  \providecommand\transparent[1]{%
    \errmessage{(Inkscape) Transparency is used (non-zero) for the text in Inkscape, but the package 'transparent.sty' is not loaded}%
    \renewcommand\transparent[1]{}%
  }%
  \providecommand\rotatebox[2]{#2}%
  \ifx\svgwidth\undefined%
    \setlength{\unitlength}{89.80767717bp}%
    \ifx\svgscale\undefined%
      \relax%
    \else%
      \setlength{\unitlength}{\unitlength * \real{\svgscale}}%
    \fi%
  \else%
    \setlength{\unitlength}{\svgwidth}%
  \fi%
  \global\let\svgwidth\undefined%
  \global\let\svgscale\undefined%
  \makeatother%
  \begin{picture}(1,0.95802915)%
    \put(0,0){\includegraphics[width=\unitlength,page=1]{pseudotrefoilsmoothv.pdf}}%
  \end{picture}%
\endgroup%

%% file: diagrams/pseudotrefoilsmoothhv.pdf_tex
\begingroup%
  \makeatletter%
  \providecommand\color[2][]{%
    \errmessage{(Inkscape) Color is used for the text in Inkscape, but the package 'color.sty' is not loaded}%
    \renewcommand\color[2][]{}%
  }%
  \providecommand\transparent[1]{%
    \errmessage{(Inkscape) Transparency is used (non-zero) for the text in Inkscape, but the package 'transparent.sty' is not loaded}%
    \renewcommand\transparent[1]{}%
  }%
  \providecommand\rotatebox[2]{#2}%
  \ifx\svgwidth\undefined%
    \setlength{\unitlength}{89.80767717bp}%
    \ifx\svgscale\undefined%
      \relax%
    \else%
      \setlength{\unitlength}{\unitlength * \real{\svgscale}}%
    \fi%
  \else%
    \setlength{\unitlength}{\svgwidth}%
  \fi%
  \global\let\svgwidth\undefined%
  \global\let\svgscale\undefined%
  \makeatother%
  \begin{picture}(1,0.95804287)%
    \put(0,0){\includegraphics[width=\unitlength,page=1]{pseudotrefoilsmoothhv.pdf}}%
  \end{picture}%
\endgroup%

%% file: diagrams/pseudotrefoilsmoothhh.pdf_tex
\begingroup%
  \makeatletter%
  \providecommand\color[2][]{%
    \errmessage{(Inkscape) Color is used for the text in Inkscape, but the package 'color.sty' is not loaded}%
    \renewcommand\color[2][]{}%
  }%
  \providecommand\transparent[1]{%
    \errmessage{(Inkscape) Transparency is used (non-zero) for the text in Inkscape, but the package 'transparent.sty' is not loaded}%
    \renewcommand\transparent[1]{}%
  }%
  \providecommand\rotatebox[2]{#2}%
  \ifx\svgwidth\undefined%
    \setlength{\unitlength}{89.80767717bp}%
    \ifx\svgscale\undefined%
      \relax%
    \else%
      \setlength{\unitlength}{\unitlength * \real{\svgscale}}%
    \fi%
  \else%
    \setlength{\unitlength}{\svgwidth}%
  \fi%
  \global\let\svgwidth\undefined%
  \global\let\svgscale\undefined%
  \makeatother%
  \begin{picture}(1,0.95804287)%
    \put(0,0){\includegraphics[width=\unitlength,page=1]{pseudotrefoilsmoothhh.pdf}}%
  \end{picture}%
\endgroup%

%% file: diagrams/ptrefverticalsmooth.pdf_tex
\begingroup%
  \makeatletter%
  \providecommand\color[2][]{%
    \errmessage{(Inkscape) Color is used for the text in Inkscape, but the package 'color.sty' is not loaded}%
    \renewcommand\color[2][]{}%
  }%
  \providecommand\transparent[1]{%
    \errmessage{(Inkscape) Transparency is used (non-zero) for the text in Inkscape, but the package 'transparent.sty' is not loaded}%
    \renewcommand\transparent[1]{}%
  }%
  \providecommand\rotatebox[2]{#2}%
  \ifx\svgwidth\undefined%
    \setlength{\unitlength}{57.23262301bp}%
    \ifx\svgscale\undefined%
      \relax%
    \else%
      \setlength{\unitlength}{\unitlength * \real{\svgscale}}%
    \fi%
  \else%
    \setlength{\unitlength}{\svgwidth}%
  \fi%
  \global\let\svgwidth\undefined%
  \global\let\svgscale\undefined%
  \makeatother%
  \begin{picture}(1,1.40076494)%
    \put(0,0){\includegraphics[width=\unitlength,page=1]{ptrefverticalsmooth.pdf}}%
  \end{picture}%
\endgroup%

%% file: diagrams/ptrefsingle.pdf_tex
\begingroup%
  \makeatletter%
  \providecommand\color[2][]{%
    \errmessage{(Inkscape) Color is used for the text in Inkscape, but the package 'color.sty' is not loaded}%
    \renewcommand\color[2][]{}%
  }%
  \providecommand\transparent[1]{%
    \errmessage{(Inkscape) Transparency is used (non-zero) for the text in Inkscape, but the package 'transparent.sty' is not loaded}%
    \renewcommand\transparent[1]{}%
  }%
  \providecommand\rotatebox[2]{#2}%
  \ifx\svgwidth\undefined%
    \setlength{\unitlength}{114.18582891bp}%
    \ifx\svgscale\undefined%
      \relax%
    \else%
      \setlength{\unitlength}{\unitlength * \real{\svgscale}}%
    \fi%
  \else%
    \setlength{\unitlength}{\svgwidth}%
  \fi%
  \global\let\svgwidth\undefined%
  \global\let\svgscale\undefined%
  \makeatother%
  \begin{picture}(1,1.01516024)%
    \put(0,0){\includegraphics[width=\unitlength,page=1]{ptrefsingle.pdf}}%
  \end{picture}%
\endgroup%

%% file: diagrams/ptrefdouble.pdf_tex
\begingroup%
  \makeatletter%
  \providecommand\color[2][]{%
    \errmessage{(Inkscape) Color is used for the text in Inkscape, but the package 'color.sty' is not loaded}%
    \renewcommand\color[2][]{}%
  }%
  \providecommand\transparent[1]{%
    \errmessage{(Inkscape) Transparency is used (non-zero) for the text in Inkscape, but the package 'transparent.sty' is not loaded}%
    \renewcommand\transparent[1]{}%
  }%
  \providecommand\rotatebox[2]{#2}%
  \ifx\svgwidth\undefined%
    \setlength{\unitlength}{114.18582891bp}%
    \ifx\svgscale\undefined%
      \relax%
    \else%
      \setlength{\unitlength}{\unitlength * \real{\svgscale}}%
    \fi%
  \else%
    \setlength{\unitlength}{\svgwidth}%
  \fi%
  \global\let\svgwidth\undefined%
  \global\let\svgscale\undefined%
  \makeatother%
  \begin{picture}(1,1.01516024)%
    \put(0,0){\includegraphics[width=\unitlength,page=1]{ptrefdouble.pdf}}%
  \end{picture}%
\endgroup%

%% file: diagrams/ptreftriple.pdf_tex
\begingroup%
  \makeatletter%
  \providecommand\color[2][]{%
    \errmessage{(Inkscape) Color is used for the text in Inkscape, but the package 'color.sty' is not loaded}%
    \renewcommand\color[2][]{}%
  }%
  \providecommand\transparent[1]{%
    \errmessage{(Inkscape) Transparency is used (non-zero) for the text in Inkscape, but the package 'transparent.sty' is not loaded}%
    \renewcommand\transparent[1]{}%
  }%
  \providecommand\rotatebox[2]{#2}%
  \ifx\svgwidth\undefined%
    \setlength{\unitlength}{114.18582891bp}%
    \ifx\svgscale\undefined%
      \relax%
    \else%
      \setlength{\unitlength}{\unitlength * \real{\svgscale}}%
    \fi%
  \else%
    \setlength{\unitlength}{\svgwidth}%
  \fi%
  \global\let\svgwidth\undefined%
  \global\let\svgscale\undefined%
  \makeatother%
  \begin{picture}(1,1.01516024)%
    \put(0,0){\includegraphics[width=\unitlength,page=1]{ptreftriple.pdf}}%
  \end{picture}%
\endgroup%

%% file: diagrams/trefoil.pdf_tex
\begingroup%
  \makeatletter%
  \providecommand\color[2][]{%
    \errmessage{(Inkscape) Color is used for the text in Inkscape, but the package 'color.sty' is not loaded}%
    \renewcommand\color[2][]{}%
  }%
  \providecommand\transparent[1]{%
    \errmessage{(Inkscape) Transparency is used (non-zero) for the text in Inkscape, but the package 'transparent.sty' is not loaded}%
    \renewcommand\transparent[1]{}%
  }%
  \providecommand\rotatebox[2]{#2}%
  \ifx\svgwidth\undefined%
    \setlength{\unitlength}{89.80767717bp}%
    \ifx\svgscale\undefined%
      \relax%
    \else%
      \setlength{\unitlength}{\unitlength * \real{\svgscale}}%
    \fi%
  \else%
    \setlength{\unitlength}{\svgwidth}%
  \fi%
  \global\let\svgwidth\undefined%
  \global\let\svgscale\undefined%
  \makeatother%
  \begin{picture}(1,0.95581325)%
    \put(0,0){\includegraphics[width=\unitlength,page=1]{trefoil.pdf}}%
  \end{picture}%
\endgroup%

%% file: diagrams/knot11n1.pdf_tex
\begingroup%
  \makeatletter%
  \providecommand\color[2][]{%
    \errmessage{(Inkscape) Color is used for the text in Inkscape, but the package 'color.sty' is not loaded}%
    \renewcommand\color[2][]{}%
  }%
  \providecommand\transparent[1]{%
    \errmessage{(Inkscape) Transparency is used (non-zero) for the text in Inkscape, but the package 'transparent.sty' is not loaded}%
    \renewcommand\transparent[1]{}%
  }%
  \providecommand\rotatebox[2]{#2}%
  \ifx\svgwidth\undefined%
    \setlength{\unitlength}{214.62014059bp}%
    \ifx\svgscale\undefined%
      \relax%
    \else%
      \setlength{\unitlength}{\unitlength * \real{\svgscale}}%
    \fi%
  \else%
    \setlength{\unitlength}{\svgwidth}%
  \fi%
  \global\let\svgwidth\undefined%
  \global\let\svgscale\undefined%
  \makeatother%
  \begin{picture}(1,0.79498999)%
    \put(0,0){\includegraphics[width=\unitlength,page=1]{knot11n1.pdf}}%
  \end{picture}%
\endgroup%

%% file: diagrams/knot11n1f.pdf_tex
\begingroup%
  \makeatletter%
  \providecommand\color[2][]{%
    \errmessage{(Inkscape) Color is used for the text in Inkscape, but the package 'color.sty' is not loaded}%
    \renewcommand\color[2][]{}%
  }%
  \providecommand\transparent[1]{%
    \errmessage{(Inkscape) Transparency is used (non-zero) for the text in Inkscape, but the package 'transparent.sty' is not loaded}%
    \renewcommand\transparent[1]{}%
  }%
  \providecommand\rotatebox[2]{#2}%
  \ifx\svgwidth\undefined%
    \setlength{\unitlength}{214.62014059bp}%
    \ifx\svgscale\undefined%
      \relax%
    \else%
      \setlength{\unitlength}{\unitlength * \real{\svgscale}}%
    \fi%
  \else%
    \setlength{\unitlength}{\svgwidth}%
  \fi%
  \global\let\svgwidth\undefined%
  \global\let\svgscale\undefined%
  \makeatother%
  \begin{picture}(1,0.79498999)%
    \put(0,0){\includegraphics[width=\unitlength,page=1]{knot11n1f.pdf}}%
  \end{picture}%
\endgroup%